\documentclass[10pt]{article}
\usepackage{mathrsfs}
\usepackage{amsmath,amscd,amsthm,amsfonts}
\oddsidemargin=0.5cm\topmargin=-1cm \numberwithin{equation}{section}

\theoremstyle{plain}
\newtheorem{theorem}{Theorem}[section]

\newtheorem{definition}[theorem]{Definition}

\newtheorem{proposition}[theorem]{Proposition}
\newtheorem{remark}[theorem]{Remark}

\newtheorem*{Theorem 3.1$'$}{Theorem 3.1$'$}

\begin{document}
\vskip 5.1cm
\title{On an entropy of $\mathbb{Z}_+^k$-actions \footnotetext{
 It is supported by NSFC(No:11071054), the Key Project of Chinese Ministry of Education(No:211020), NCET(No:11-0935)  and the SRF for ROCS, SEM.}}
\author {Yujun Zhu and Wenda Zhang\\
\small {College of Mathematics and Information Science, and}\\
\small {Hebei Key Laboratory of Computational Mathematics and Applications}\\
\small {Hebei Normal University, Shijiazhuang, 050024, P.R.China}}
\date{}
\maketitle
\begin{center}
\begin{minipage}{130mm}
{\bf Abstract}: In this paper, a definition of entropy for
$\mathbb{Z}_+^k(k\geq2)$-actions due to S. Friedland \cite{Friedland} is
studied. Unlike the traditional definition, it may take a nonzero value for
actions whose generators have finite (even zero) entropy as single
transformations. Some basic properties are investigated and its
value for the $\mathbb{Z}_+^k$-actions on circles generated
by expanding endomorphisms is given. Moreover, an upper bound of this entropy for the $\mathbb{Z}_+^k$-actions on tori generated
by expanding endomorphisms is obtained via the
preimage entropies, which are entropy-like invariants depending on
the ``inverse orbits" structure of the system.

\end{minipage}
\end{center}
\section{Introduction}

Based on the need in the study of lattice statistical mechanics,
Ruelle \cite{Ruelle} introduced the concept of entropy for
$\mathbb{Z}^k(k\geq2)$-actions. A necessary condition for this entropy to be
positive is that the generators should have infinite entropy as
single transformations. In \cite{Friedland}, Friedland gave a new
definition of entropy for $\mathbb{Z}^k$-actions (or, more
generally, $\mathbb{Z}_+^k$-actions, here $\mathbb{Z}_+=\{0,1,2,\cdots\}$) which is appropriate for that
whose generators have finite entropy as single transformations.

We begin by recalling the definition of topological entropy for
$\mathbb{Z}_+$-actions. Let $(X, d_X)$ be a compact metric space and
$C^0(X,X)$  the set of continuous maps on $X$. Any $f\in C^0(X,X)$
naturally generates a  $\mathbb{Z}_+$-action: $\mathbb{Z}_+
\longrightarrow C^0(X,X), n\longmapsto f^n$. Let $K$ be a compact
subset of $X$. For any $\varepsilon >0$, a subset $E\subset X$ is
said to be an $(f,n,\varepsilon)$-spanning set of $K$, if for any
$x\in K$, there exists $y\in E$ such that
$$
\max_{0\leq i\leq
n-1}d_X(f^i(x),f^i(y))\leq \varepsilon.
$$
Let $r_{d_X}(f,n,\varepsilon,K)$ denote the smallest cardinality of
any $(f,n,\varepsilon)$-spanning set of $K$. A subset $F\subset K$
is said to be an $(f,n,\varepsilon)$-separated set of $K$, if
$x,y\in F,x\neq y$, implies
$$
\max_{0\leq i\leq
n-1}d_X(f^i(x),f^i(y))> \varepsilon.
$$
Let $s_{d_X}(f,n,\varepsilon,K)$ denote the largest cardinality of
any $(f,n,\varepsilon)$-separated set of $K$. Let
$$
h(f,K)=\lim_{\varepsilon\rightarrow
0}\limsup_{n\rightarrow\infty}\frac{1}{n}\log
\;s_{d_X}(f,n,\varepsilon,K).
$$
By a standard discussion, we can give $h(f,K)$ using spanning set,
i.e., we can replace $s_{d_X}(f,n,\varepsilon,K)$ by
$r_{d_X}(f,n,\varepsilon,K)$ in the above equation. The
\emph{topological entropy} of $f$ is defined by $h(f)=h(f,X)$.

Now we recall the traditional definition and Friedland's definition
of entropy for $\mathbb{Z}_+^k$-actions.
 Let $(X, d_X)$ be a compact metric space and
$T:\mathbb{Z}_+^k\longrightarrow C^0(X,X)$  a continuous
$\mathbb{Z}_+^k$-action on $X$.  Denote the cube
$\prod\limits_{i=1}^{k}\{0,\cdots,n-1\}\subset\mathbb{Z}_+^k$ by
$Q_n$. A set $E\subset X$ is $(T, n,\varepsilon)$-spanning  if
for every $x\in X$ there exists a $y\in E$ with
$d_X\bigl(T^{\bf{i}}(x), T^{\bf{i}}(y)\bigr)\leq\varepsilon$ for all
${\bf{i}}\in Q_n$. Let $r_{d_X}(T, n,\varepsilon,X)$ be the smallest
cardinality of any $(T, n,\varepsilon)$-spanning set. A set
$F\subset X$ is an $(T, n,\varepsilon)$-separated set if for any
$x,y\in F$, $x\neq y$ implies $d_X\bigl(T^{\bf{i}}(x),
T^{\bf{i}}(y)\bigr)>\varepsilon$ for some ${\bf{i}}\in Q_n$. Let
$s_{d_X}(T, n,\varepsilon,X)$ be the largest cardinality of any
$(T, n,\varepsilon)$-separated set. The traditional definition
of $\tilde{h}(T)$ is given by
\begin{equation}\label{entropy}
\tilde{h}(T)=\lim_{\varepsilon\rightarrow
0}\limsup_{n\rightarrow\infty}\frac{1}{n^k}\log
\;s_{d_X}(T, n,\varepsilon,X).
\end{equation}
By a standard discussion, we can give $\tilde{h}(T)$ using spanning
set, i.e., we can replace $s_{d_X}(T, n,\varepsilon,X)$ by
$r_{d_X}(T, n,\varepsilon,X)$ in (\ref{entropy}). For the general theory of entropy of $\mathbb{Z}^k$-actions, see Schmidt's comprehensive book \cite{Schmidt}, and for the more general theory of entropy for countable amenable group actions, see for example \cite{Ornstein} and \cite{Huang}.

It is well known that a necessary condition for $\tilde{h}(T)$ to be
positive is that the generators
$\{T_i:=T(0,\cdots,0,1_{(i)},0,\cdots,0)\}_{i=1}^k$ should have
infinite entropy as single transformations. In contrast to the
traditional definition, Friedland \cite{Friedland} introduced another
definition of the topological entropy as follows. Define the \emph{orbit space} of $T$ by
$$
X_T=\bigl\{\bar{x}=\{x_n\}_{n\in
{\mathbb{Z}}_+}\in\prod_{n\in{\mathbb{Z}}_+}X : \text{ for any }  n\in{\mathbb{Z}}_+, T_{i_n}(x_n)=x_{n+1}
\text{ for some } T_{i_n}\in\{T_i\}_{i=1}^k \bigr\}.
$$
This is a closed subset of the compact space
$\prod\limits_{n\in{\mathbb{Z}}_+}X$ and so is again compact. A
natural metric on $X_T$ is defined by
\begin{equation}\label{metric}
d_{X_T}\bigl(\bar{x},\bar{y}\bigr)=\sum_{n=0}^\infty\frac{d_X(x_n,y_n)}{2^n}
\end{equation}
for $\bar{x}=\{x_n\}_{n\in
{\mathbb{Z}}_+}, \bar{y}=\{y_n\}_{n\in
{\mathbb{Z}}_+}\in X_T$.
We can define a natural shift map $\sigma_T: X_T\rightarrow X_T$ by
$\sigma_T(\{x_n\}_{n\in {\mathbb{Z}}_+})=\{x_{n+1}\}_{n\in
{\mathbb{Z}}_+}$. Thus we have associated in a natural way a
${\mathbb{Z}}_+$-action with the ${\mathbb{Z}}_+^k$-action.
\begin{definition}
We define the \emph{topological entropy} $h(T)$ of the
${\mathbb{Z}}_+^k$-action $T$ to be the topological entropy of the
map $\sigma_T: X_T\rightarrow X_T$, i.e.,
$$
h(T)=h(\sigma_T)=\lim_{\varepsilon\rightarrow
0}\limsup_{n\rightarrow\infty}\frac{1}{n}\log
\;s_{d_{X_T}}(\sigma_T, n,\varepsilon,X_T),
$$
where $s_{d_{X_T}}(\sigma_T, n,\varepsilon,X_T)$ is the largest cardinality of any
$(\sigma_T, n,\varepsilon)$-separated set in $X_T$.
(Similarly, we can replace $s_{d_{X_T}}(\sigma_T, n,\varepsilon,X_T)$ by $r_{d_{X_T}}(\sigma_T, n,\varepsilon,X_T)$,
the smallest cardinality of any $(\sigma_T, n,\varepsilon)$-spanning set in $X_T$.)
\end{definition}

The main purpose of this paper is to investigate some fundamental
properties of the entropy $h(T)$ of ${\mathbb{Z}}_+^k$-actions and
evaluate its values for some standard examples.

In section 2, some basic properties of $h(T)$ are given. It is well
known that for any map $f:X\rightarrow X$, the power rule for its
entropy holds, i.e., $h(f^m)=mh(f)$ for any positive integer $m$.
However, we can only get $h(T^m)\leq mh(T)$ for any
$\mathbb{Z}_+^k$-action $T$ (Proposition \ref{prop1}). We also show
in Proposition \ref{prop1} that the entropy of any subgroup action
(especially, each generator) of $T$ is less than or equal to that of
$T$. When each generator $T_i$ is Lipschitzian with Lipschitz
constant $L(T_i)$, then we can get an upper bound of $h(T)$ by
$\log\sum\limits_{i=1}^kL_+(T_i)^{D(x)}$, where $L_+(T_i)=\max\{1,
L(T_i)\}$ and $D(x)$ is the ball dimension of $X$ (Proposition
\ref{prop2}). The entropy of a skew product transformation which is
an extension of $(X_T, \sigma_T)$ is also considered (Proposition
\ref{prop3}).

In Section 3, we use Gellar and Pollicott's method \cite{Geller} to show  (in Theorem 3.1) that for the
$\mathbb{Z}_+^k$-action $T$ on the unit circle $X=\mathbb{S}^1$
generated by pairwise different endomorphisms $T_i(x)=L_i
x(\text{mod } 1)$, where $L_i, 1\leq i\leq k$ are all positive
integers greater than $1$,
\begin{equation}\label{circleentropy}
h(T)=\log \sum\limits_{i=1}^kL_i.
\end{equation}

In Section 4, we use other entropy-like invariants, the so called
preimage entropies which rely on the preimage structure of the
system,  to show (in Theorem 4.1) that for the $\mathbb{Z}_+^k$-action $T$ on the torus $\mathbb{T}^n$
generated by pairwise different matrices $\{A_i\}_{i=1}^k$ whose
eigenvalues $\{\lambda_i^{(1)}, \cdots, \lambda_i^{(n)}\}_{i=1}^k$
are of modulus greater than $1$,
\begin{equation}\label{torusentropy}
h(T)\leq\log\Bigl(\sum_{i=1}^k\prod_{j=1}^n|\lambda_i^{(j)}|\Bigr).
\end{equation}

\section{Some basic properties of $h(T)$}\label{section2}
Throughout this section we always assume that $T:\mathbb{Z}_+^k\longrightarrow C^0(X,X)$ is a continuous $\mathbb{Z}_+^k$-action with the generators $\{T_i : 1\leq i\leq k\}$.

It is well known that topological entropy of a map (i.e., a
$\mathbb{Z}_+$-action) is invariant under conjugacy. Now we can show
that for any $\mathbb{Z}_+^k$-action a similar property holds true.
We call another $\mathbb{Z}_+^k$-action $T'$ is \emph{topologically
conjugate to} $T$, if their generators $\{T_i : 1\leq i\leq k\}$ and
$\{T'_i : 1\leq i\leq k\}$ are pairwise conjugate under the same
homeomorphism $h$, i.e. we have the following commutative diagrams
\begin{equation*}
\begin{CD}
X @>T_{i}>>  X\\
@VhVV   @VVhV\\
X @>T'_{i}>>  X
\end{CD}
\end{equation*}
for each $i$.  Since  we can express $T'_{i}$ by $hT_{i}h^{-1}$, by
the definition of the entropy we can get the following property
immediately.

\begin{proposition}\label{invariant}
Let $T$ and $T'$ be two conjugate  $\mathbb{Z}_+^k$-actions with generators
$\{T_i : 1\leq i\leq
k\}$ and $\{T'_i : 1\leq i\leq
k\}$ respectively, then
$$
h(T)=h(T').
$$
\end{proposition}

The following proposition concerns the relation between the entropy of the power $T^m$ and that of $T$, and the relation between the entropy of the subgroup action $T^{(l)}$  and that of $T$.

\begin{proposition}\label{prop1}
Let $T:\mathbb{Z}_+^k\longrightarrow C^0(X,X)$ be a continuous
$\mathbb{Z}_+^k$-action with the generators $\{T_i : 1\leq i\leq
k\}$. We have the following properties of the entropy $h(T)$.

(1) For $m>1$, we have $h(T^m)\leq mh(T)$, where $T^m$ is the
${\mathbb{Z}}_+^k$-action with the generators $\{T_i^m : 1\leq
i\leq k\}$.

(2) For any $1\leq l<k$ and any ${\mathbb{Z}}_+^l$-action
$T^{(l)}$ generated by some subcollection
$\{T_{i_1},\cdots,T_{i_l}\}\subset\{T_1,\cdots,T_k\}$, we have
$$
h(T^{(l)})\leq h(T).
$$
In particular, $h(T_i)\leq h(T)$ for any $1\leq i\leq k$.

\end{proposition}

\begin{proof}

(1) Let
$\tilde{X}=\bigl\{\{x_n\}_{n\in{\mathbb{Z}}_+}\in\prod\limits_{n\in{\mathbb{Z}}_+}X\;:\;
\text{for any} \;j\in{\mathbb{Z}}_+, \text{there exists some}
\;1\leq i\leq k, \;\text{such that for all} \;0\leq s\leq m,
x_{jm+s}=T_i^s(x_{jm})\bigr\}$. It is obvious that $\tilde{X}\subset
X_T$. Define $\pi: \tilde{X}\rightarrow X_{T^m}$ by
$\pi\bigl(\{x_n\}_{n\in{\mathbb{Z}}_+}\bigr)=\{x_{nm}\}_{n\in{\mathbb{Z}}_+}$.
It is easy to obtain that $\pi$ is continuous and
$$
\sigma_{T^m}\circ\pi\bigl(\{x_n\}_{n\in{\mathbb{Z}}_+}\bigr)=\pi\circ\sigma_T^m\bigl(\{x_n\}_{n\in{\mathbb{Z}}_+}\bigr)
$$
for any $\{x_n\}_{n\in{\mathbb{Z}}_+}\in \tilde{X}$. Therefore,
$$
h(\sigma_{T^m})\leq h(\sigma_{T}^m|_{\tilde{X}})\leq
h(\sigma_{T}^m)=mh(\sigma_{T}),
$$
in which the last equality is from the well known power rule for
topological entropy (see \cite{Walters}, for example).

(2) For the ${\mathbb{Z}}_+^l$-action $T^{(l)}$ generated by some
 $\{T_{i_1},\cdots,T_{i_l}\}$, denote
$$
X_{T^{(l)}}=\bigl\{\{x_n\}_{n\in{\mathbb{Z}}_+}\in\prod\limits_{n\in{\mathbb{Z}}_+}X\;:\;
x_{n+1}=T_{i_j}(x_n)\; \text{for some} \;1\leq j\leq l\bigr\}.
$$
It is obvious that
$$
h(T^{(l)})=h(\sigma_{T^{(l)}})=h(\sigma_T|_{X_{T^{(l)}}})\leq
h(\sigma_{T})=h(T).
$$
\end{proof}

\begin{remark}
For (1) of Proposition \ref{prop1}, either of the equality and strictly inequality
in ``$h(T^m)\leq mh(T)$" can possibly hold. For example, for the $\mathbb{Z}_+^k$-action $T$ in Theorem 3.1, $h(T^m)=mh(T)$;
for the $\mathbb{Z}_+^k$-action $T$ on the unit circle $\mathbb{S}^1$ whose generators $\{T_i: 1\leq i\leq k\}$ are pairwise different rotations, by Theorem 4 of \cite{Geller} we have $h(T^m)=h(T)=\log k< mh(T)$.

From (2) of Proposition \ref{prop1}, for any $\mathbb{Z}_+^k$-action
$T$ with generators $\{T_i: 1\leq i\leq k\}$, the entropy of each
generator is less than or equal to that of $T$, i.e., $h(T_i)\leq
h(T), 1\leq i\leq k$. In general, $h(T_i)$ is strict less than
$h(T)$, even for the actions with some trivial generators. For example, for the above $\mathbb{Z}_+^k$-action $T$ on the unit circle $\mathbb{S}^1$ whose generators are pairwise different rotations, it is obvious that $h(T_i)=0$ for each $i$, but $h(T)=\log k$.

However, for any $\mathbb{Z}_+^k$-action $T$ with positive traditional entropy, i.e., $\tilde{h}(T)>0$, such as the full $k$-dimensional $m$-shift transformation $T$ on the space
$$
X=\{0,\cdots,m-1\}^{\mathbb{Z}_+^k}=\bigl\{\{x_{(i_1,\cdots,i_k)}\}_{(i_1,\cdots,i_k)\in{\mathbb{Z}}_+^k}\in\prod\limits_{(i_1,\cdots,i_k)\in{\mathbb{Z}}_+^k}\{0,
1, \cdots, m-1\}\bigr\}
$$
generated by
$$
T_j: \{x_{(i_1,\cdots,i_k)}\}
\mapsto\{x_{(i_1,\cdots,i_{j}+1,\cdots,i_k)}\}, \;\;1\leq j\leq k,
$$
we have $h(T)=h(T_i)=\infty$ for any $1\leq i\leq k$.
\end{remark}

Let $(X,d)$ be a compact metric space and $b(\varepsilon)$ the
minimum cardinality of covering of $X$ by $\varepsilon$-balls. Then
$$
D(X)=\limsup_{\varepsilon\rightarrow0}\frac{{\log}b(\varepsilon)}{|{\log}\varepsilon|}\in\mathbb{R}\cup\{\infty\}
$$
is called the \emph{ball dimension} of $X$.
It is well known that if a map $f:X\longrightarrow X$ is Lipschitzian with Lipschitz constant $L(f)$, then
$$
h(f)\leq D(X)\log(\max\{1, L(f)\}),
$$
see Theorem 3.2.9 of \cite{Katok} for example. In the following, we give the corresponding inequalities for  ${\mathbb{Z}}_+^k$-actions.

\begin{proposition}\label{prop2}
Let $T:\mathbb{Z}_+^k\longrightarrow C^0(X,X)$ be a continuous
$\mathbb{Z}_+^k$-action with the generators $\{T_i : 1\leq i\leq
k\}$. If  the ball dimension of $X$ is finite, i.e., $D(X)<\infty$, and each $T_i:
X\rightarrow X$ is Lipschitzian with Lipschitz constant $L(T_i)$, then
$$
h(T)\leq\log\sum_{i=1}^{k}L_+(T_i)^{D(X)},
$$
where $L_+(T_i)=\max\{1, L(T_i)\}$.

In particular, if $X$ is an $m$-dimensional compact Riemannian
manifold and each $T_i,1\leq i\leq k$, is differentiable, then
$$
h(T)\leq \log \sum_{i=1}^{k}\max\{1,\sup_{ x\in X}\|d_xT_i\|^m\}.
$$
\end{proposition}

\begin{proof}
It is well known that the topological entropy is unchanged by taking
uniformly equivalent metrics (Theorem 7.4 of \cite{Walters}). Here
we say two metrics $d_X$ and $d'_X$ on $X$ are uniformly equivalent
if
$$
id: (X,d_X)\longrightarrow (X,d'_X)\;\;\mbox{and}\;\; id: (X,d'_X)\longrightarrow (X,d_X)
$$
are both uniformly continuous.

Take $\rho>\max\limits_{1\leq i\leq k}L_+(T_i)$. Clearly, $\rho>1$.
Now we define two metrics $\tilde{d}_{X_T}$ and $d'_{X_T}$ on $X_T$
by
\begin{equation}\label{metric2}
\tilde{d}_{X_T}\bigl(\bar{x},\bar{y}\bigr)=\sum_{n=0}^\infty\frac{d_X(x_n,y_n)}{\rho^n}\;\;\mbox{and}\;\; d'_{X_T}(\bar{x}, \bar{y})=\sup_{n\geq
0}\frac{d_X(x_n,y_n)}{\rho^n}
\end{equation}
for any $\bar{x}=\{x_n\}_{n\in
{\mathbb{Z}}_+}, \bar{y}=\{y_n\}_{n\in
{\mathbb{Z}}_+}$. Since $\rho>1$, the metrics $\tilde{d}_{X_T}$ and $d'_{X_T}$ are both
uniformly equivalent to $d_{X_T}$ which is defined in (\ref{metric}).

In the following we will estimate the entropy $h(T)$ with respect to
the metric $d'_{X_T}$. For any $\varepsilon>0$, consider a maximal
$(\sigma_T, m, \varepsilon)$-separated set $E$ of  $X_T$ with cardinality
$s_{d'_{X_T}}(\sigma_T, m, \varepsilon, X_T)$. Obviously, for any
$\bar{x}=\{x_n\}_{n\in {\mathbb{Z}}_+}, \bar{y}=\{y_n\}_{n\in
{\mathbb{Z}}_+}\in E$ with $\bar{x}\neq\bar{y}$, we have
$$
\max_{0\leq s\leq m-1}\sup_{n\geq
s}\frac{d_X(x_n,y_n)}{\rho^{n-s}}>\varepsilon.
$$
Let
$K(\varepsilon)=\big\lceil\log_\rho\frac{\mbox{diam}(X,d_X)}{\varepsilon}\big\rceil$.
In order to estimate the cardinality of $E$, we will write it into
the union of subsets which consists of the points in $E$ with the
first $m+K(\varepsilon)$ elements lie in the same orbit space of
some sequence of $(T_{i_1}, T_{i_2}, \cdots,
T_{i_{m+K(\varepsilon)-1}})$. From the definition of $d'_{X_T}$ and
the choice of $K(\varepsilon)$, we can estimate the cardinality of
each of these subsets easily.

For any $(i_1,\cdots,i_{m+K(\varepsilon)})\in
\prod\limits_{n=1}^{m+K(\varepsilon)}\{1,\cdots,k\}$, denote
$$
\tilde{X}_{(i_1,\cdots,i_{m+K(\varepsilon)})}=\bigl\{\bar{x}=\{x_n\}_{n\in{\mathbb{Z}}_+}\;:\;x_n=T_{i_n}(x_{n-1})\;
\mbox{for}\; 1\leq n\leq m+K(\varepsilon)-1\bigr\}
$$
and
$$
\tilde{E}_{(i_1,\cdots,i_{m+K(\varepsilon)})}=E\cap\tilde{X}_{(i_1,\cdots,i_{m+K(\varepsilon)})}.
$$
Clearly,
$$
X_T=\bigcup_{(i_1,\cdots,i_{m+K(\varepsilon)})\in
\prod\limits_{n=1}^{m+K(\varepsilon)}\{1,\cdots,k\}}\tilde{X}_{(i_1,\cdots,i_{m+K(\varepsilon)})}
$$
and
$$
E=\bigcup_{(i_1,\cdots,i_{m+K(\varepsilon)})\in
\prod\limits_{n=1}^{m+K(\varepsilon)}\{1,\cdots,k\}}\tilde{E}_{(i_1,\cdots,i_{m+K(\varepsilon)})}.
$$
(Note that, each of them may be not a disjoint union.) Moreover, by
the choice of $\rho$ and $K(\varepsilon)$, we can see that for any
$\overline{x}=\{x_n\}_{n\in{\mathbb{Z}}_+},
\overline{y}=\{y_n\}_{n\in{\mathbb{Z}}_+}\in
\tilde{E}_{(i_1,\cdots,i_{m+K(\varepsilon)})}$, if $x_n=y_n$ for any
$0\leq n\leq m+K(\varepsilon)-1$ then $\overline{x}=\overline{y}$.
Therefore, if we denote the projection from
$\prod\limits_{n\in{\mathbb{Z}}_+}X$ to its factor
$\prod\limits_{n=0}^{m+K(\varepsilon)-1}X$ by
$\mbox{Proj}_{m+K(\varepsilon)}$ and let
$$
E_{(i_1,\cdots,i_{m+K(\varepsilon)})}=\mbox{Proj}_{m+K(\varepsilon)}(\tilde{E}_{(i_1,\cdots,i_{m+K(\varepsilon)})}),
$$
then
$$
\mbox{card}
(E_{(i_1,\cdots,i_{m+K(\varepsilon)})})=\mbox{card}
(\tilde{E}_{(i_1,\cdots,i_{m+K(\varepsilon)})}).
$$
Define a metric $d_{(i_1,\cdots,i_{m+K(\varepsilon)})}$ on
$X_{(i_1,\cdots,i_{m+K(\varepsilon)})}:=\mbox{Proj}_{m+K(\varepsilon)}(\tilde{X}_{(i_1,\cdots,i_{m+K(\varepsilon)})})$
 by
$$
d_{(i_1,\cdots,i_{m+K(\varepsilon)})}(\bar{z},\bar{z}')=\max_{0\leq s\leq
m+K(\varepsilon)-1}\frac{d_X(z_s,z'_s)}{\rho^s}
$$
for any $\bar{z}=\{z_s\}_{s=0}^{m+K(\varepsilon)-1}$ and
$\bar{z}'=\{z'_s\}_{s=0}^{m+K(\varepsilon)-1}\in X_{(i_1,\cdots,i_{m+K(\varepsilon)})}$. From the choice of $\rho$, we can see that
$(X_{(i_1,\cdots,i_{m+K(\varepsilon)})},
d_{(i_1,\cdots,i_{m+K(\varepsilon)})})$ is isometric to $(X,
d_{X})$. That is,
$$
d_{(i_1,\cdots,i_{m+K(\varepsilon)})}(\bar{z},\bar{z}')=d_{X}(z_0, z'_0).
$$
Then the ball dimension of $(X_{(i_1,\cdots,i_{m+K(\varepsilon)})},
d_{(i_1,\cdots,i_{m+K(\varepsilon)})})$ is equal to $D(X)$.

Let
$\varepsilon_{(i_1,\cdots,i_{m+K(\varepsilon)})}=\displaystyle\frac{\varepsilon}{3\prod\limits_{n=1}^{m+K(\varepsilon)-1}L_+(T_{i_n})}$.
Then for any $\bar{z}, \bar{z}'\in
E_{(i_1,\cdots,i_{m+K(\varepsilon)})}$  with $\bar{z}\neq\bar{z}'$, we have
$$
B_{d_{(i_1,\cdots,i_{m+K(\varepsilon)})}}(\bar{z},
\varepsilon_{(i_1,\cdots,i_{m+K(\varepsilon)})})\cap
B_{d_{(i_1,\cdots,i_{m+K(\varepsilon)})}}(\bar{z}',
\varepsilon_{(i_1,\cdots,i_{m+K(\varepsilon)})})=\emptyset.
$$
Therefore,
$$
\mbox{card}
(E_{(i_1,\cdots,i_{m+K(\varepsilon)})})\leq\frac{\alpha}{(\varepsilon_{(i_1,\cdots,i_{m+K(\varepsilon)})})^{D(X)}},
$$
where $\alpha$ is a constant independent of
$(i_1,\cdots,i_{m+K(\varepsilon)})$. Hence
$$
\begin{aligned}
s_{d'_{X_T}}(\sigma_T, m, \varepsilon, X_T)=&\;\mbox{card}(E)\\
\leq&\sum_{(i_1,\cdots,i_{m+K(\varepsilon)})\in
\prod\limits_{n=1}^{m+K(\varepsilon)}\{1,\cdots,k\}}\mbox{card}
(E_{(i_1,\cdots,i_{m+K(\varepsilon)})})\\
\leq&\sum_{(i_1,\cdots,i_{m+K(\varepsilon)})\in
\prod\limits_{n=1}^{m+K(\varepsilon)}\{1,\cdots,k\}}\frac{\alpha\cdot
3^{D(X)}}{\varepsilon^{D(X)}}\cdot\prod\limits_{n=1}^{m+K(\varepsilon)-1}L_+(T_{i_n})^{D(X)}\\
=&\frac{\alpha\cdot
3^{D(X)}}{\varepsilon^{D(X)}}\cdot\Bigl(\sum_{i=1}^kL_+(T_{i})^{D(X)}\Bigr)^{m+K(\varepsilon)-1}.
\end{aligned}
$$
Thus,
$$
h(T)=\lim_{\varepsilon\longrightarrow 0}\limsup_{m\rightarrow\infty}\frac{1}{m}\log s_{d'_{X_T}}(\sigma_T, m, \varepsilon, X_T)\leq\log\sum_{i=1}^kL_+(T_{i})^{D(X)}.
$$
\end{proof}

In the following, we will consider a skew product transformation
such that $(\sigma_T, X_T)$ is its factor, and we will use it to
evaluate the entropy of $\mathbb{Z}_+^k$-action on circles in the
next section.

Let $\Sigma_k=\prod\limits_{n\in {\mathbb{Z}}_+}\{1,\cdots,k\}$ be the standard symbolic space with the product topology. A natural metric $d_{\Sigma_k}$ on $\Sigma_k$ is defined by
$$
d_{\Sigma_k}\bigl(\{i_n\}_{n\in{\mathbb{Z}}_+},\{j_n\}_{n\in{\mathbb{Z}}_+}\bigr)=
\sum_{n=0}^\infty\frac{d(i_n,j_n)}{2^n}
$$
for $\{i_n\}_{n\in{\mathbb{Z}}_+},\{j_n\}_{n\in{\mathbb{Z}}_+}\in \Sigma_k$, where $d(i_n,j_n)=0$ when $i_n=j_n$, and $d(i_n,j_n)=1$ when $i_n\neq j_n$. Let $Y=\Sigma_k\times X$ be endowed with the product topology
and define a map $\tilde{\sigma}: Y\rightarrow Y$ by
$$
\tilde{\sigma}\bigl(\{i_n\}_{n\in{\mathbb{Z}}_+},x\bigr)=\bigl(\{i_{n+1}\}_{n\in{\mathbb{Z}}_+},
T_{i_0}x\bigr).
$$
This is a skew product over the shift transformation
$$
\sigma_k: \Sigma_k\rightarrow\Sigma_k,\;\;\;\;
\{i_n\}_{n\in{\mathbb{Z}}_+}\mapsto\{i_{n+1}\}_{n\in{\mathbb{Z}}_+}.
$$
The basis transformation $\sigma_k$ is a natural factor of this skew product, hence by Bowen's entropy inequality in \cite{Bowen},  we have that
\begin{equation}\label{Bowen}
h(\tilde{\sigma})\leq h(\sigma_k) +\sup_{\{i_n\}_{n\in{\mathbb{Z}}_+}\in\sum_k}
h(\tilde{\sigma},Y_{\{i_n\}_{n\in{\mathbb{Z}}_+}}),
\end{equation}
where
$Y_{\{i_n\}_{n\in{\mathbb{Z}}_+}}=\bigl\{(\{i_n\}_{n\in{\mathbb{Z}}_+},x)\in
Y:x\in X\bigr\}$. Clearly, $h(\sigma_k)=\log k$. Moreover, if the ball dimension of $X$ is finite, i.e., $D(X)<\infty$,
and each $T_i:
X\rightarrow X$ is Lipschitzian with Lipschitz constant $L(T_i)$, then from the proof of Proposition \ref{prop2} we have that
$$
h(\tilde{\sigma}, Y_{\{i_n\}_{n\in{\mathbb{Z}}_+}})\leq\limsup\limits_{n\rightarrow\infty}\displaystyle
\frac{1}{n}\log\prod\limits_{j=1}^nL_+(T_{i_j})^{D(X)}\leq D(X)\log L,
$$
in which $L=\max\limits_{1\leq i\leq k}L_+(T_i)$,
and hence
\begin{equation}\label{Bowen1}
h(\tilde{\sigma})\leq \log k
+D(X)\log L.
\end{equation}

In the following we can see that $\sigma_T$ is another factor of $\tilde{\sigma}$.

\begin{proposition}\label{prop3}
Let $T:\mathbb{Z}_+^k\longrightarrow C^0(X,X)$ be a continuous
$\mathbb{Z}_+^k$-action with the generators $\{T_i : 1\leq i\leq
k\}$ and $\tilde{\sigma}: Y\rightarrow Y$ be as above.  Then $\tilde{\sigma}$ is an extension of $\sigma_T$, and hence
$$
h(T)=h(\sigma_T)\leq h(\tilde{\sigma}).
$$

\end{proposition}

\begin{proof}
Define a map $\pi: Y\rightarrow X_T$ by
$$
\tilde{\pi}\bigl((\{i_n\}_{n\in{\mathbb{Z}}_+},x)\bigr)=
\{x_n\}_{n\in{\mathbb{Z}}_+},
$$
where $x_0=x$ and $x_n=T_{i_{n-1}}\circ\cdots\circ T_{i_1}\circ T_{i_0}(x)$  for $n\geq 1$.
We claim that it is a semi-conjugacy between $\tilde{\sigma}$ and $\sigma_T$. In fact, we firstly have that $\pi\circ\tilde{\sigma}=\sigma_T\circ\pi$ from the definitions of $\pi, \tilde{\sigma}$ and $\sigma_T$. Secondly,
$\pi$ is surjective since for any point $\{x_n\}_{n\in{\mathbb{Z}}_+}\in X_T$
we can construct $\bigl(\{i_n\}_{n\in{\mathbb{Z}}_+},x\bigr)\in Y$
by setting $x=x_0$ and choosing $i_n$ for $n\geq 0$ inductively such that $i_n=j$ if $x_{n+1}=T_j(x_n)$ for some $1\leq j\leq k$ (Please note that the choice of $i_n$ may not be unique). Finally, $\pi$ is continuous since for any sequence of points $\{\overline{y}^{(i)}=\bigl(\{i_n^{(i)}\}_{n\in{\mathbb{Z}}_+},x^{(i)}\bigr)\}_{i\in{\mathbb{Z}}_+}$ tends to $\overline{y}=\bigl(\{i_n\}_{n\in{\mathbb{Z}}_+},x\bigr)$ as $i\longrightarrow \infty$, we have $\{i_n^{(i)}\}_{n\in{\mathbb{Z}}_+}\longrightarrow \{i_n\}_{n\in{\mathbb{Z}}_+}$ and $x^{(i)}\longrightarrow x$ as $i\longrightarrow \infty$, and hence from the definition of $\pi$, the uniform continuity of $T_i, 1\leq i\leq k$, and the topologies of $Y$ and $X_T$, $\pi(\overline{y}^{(i)})\longrightarrow \pi(\overline{y})$ as $i\longrightarrow \infty$.
Therefore, $\sigma_T$ is a factor of $\tilde{\sigma}$, and hence $h(T)=h(\sigma_T)\leq h(\tilde{\sigma}).$
\end{proof}

\section{$\mathbb{Z}_+^k$-actions on circles generated by expanding endomophisms}

Let $\mathbb{S}^1=\mathbb{R}/\mathbb{Z}$ be the unit circle with the ``geodesic" metric $d_{\mathbb{S}^1}$, i.e., for any $x, y\in\mathbb{S}^1$, $d_{\mathbb{S}^1}(x, y)$ is the length of the shorter path joining them.

As one of the simplest system
$f:\mathbb{S}^1\longrightarrow\mathbb{S}^1,x\longmapsto px(\mbox{mod
} 1)$, its entropy  $h(f)=\log p$.  However, for
$\mathbb{Z}_+^k$-action on circles which is generated by this kind
of endomorphisms, its entropy is not easy to be calculated. In
\cite{Friedland}, Friedland conjectured that for a
$\mathbb{Z}_+^2$-action on the circle
$\mathbb{S}^1$ whose generators are
$T_1=px(\mbox{mod } 1)$ and $T_2=qx(\mbox{mod } 1)$, where $p$ and
$q$ are two co-prime integers, its entropy $h(T)=\log(p+q)$. Soon
afterwards Geller and Pollicott \cite{Geller} answered this
conjecture affirmatively under a weaker condition ``$p, q$ are all
integers greater than 1". In this section we generalize the main result in \cite{Geller} to
$\mathbb{Z}_+^k$-action on circles.

\begin{theorem}\label{theorem1}
Let $T:\mathbb{Z}_+^k\longrightarrow C^0(\mathbb{S}^1,\mathbb{S}^1)$
be a continuous $\mathbb{Z}_+^k$-action on the circle
$X=\mathbb{S}^1$ with the generators $T_i(1\leq i\leq k)$ defined by
$T_i(x)=L_ix(\emph{mod } 1)$ where $L_i, 1\leq i\leq k,$ are all
integers greater than $1$ and are pairwise different. Then the
formula (\ref{circleentropy}) holds, i.e.,
$$
h(T)=\log \sum\limits_{i=1}^kL_i.
$$
\end{theorem}

\begin{proof}
As we have done in Proposition \ref{prop3}, denote
$Y=\Sigma_k\times X$ endowed with the product topology
and define a map $\tilde{\sigma}: Y\rightarrow Y$ by
$$
\tilde{\sigma}\bigl(\{i_n\}_{n\in{\mathbb{Z}}_+},x\bigr)=\bigl(\{i_{n+1}\}_{n\in{\mathbb{Z}}_+},
T_{i_0}x\bigr).
$$
This is a skew product over the shift transformation
$$
\sigma_k: \Sigma_k\rightarrow\Sigma_k,\;\;\;\;
\{i_n\}_{n\in{\mathbb{Z}}_+}\mapsto\{i_{n+1}\}_{n\in{\mathbb{Z}}_+}.
$$

Consider a cover of $\Sigma_k\times\mathbb{S}^1$ consisting of the
closed sets
$$
\Bigl\{[i]_0\times\Bigl[\frac{j}{M}, \frac{j+1}{M}\Bigr] \;:\; 1\leq
i\leq k,\; 0\leq j\leq M-1\Bigr\},
$$
where $[i]_0=\bigl\{\{i_n\}_{n\in{\mathbb{Z}}_+}\in\Sigma_k
\;:\;i_0=i\bigr\}$ and $M=\prod\limits_{i=1}^kL_i$. It is clear that this
cover consists of $kM$ elements. We label them by $\{B_l\}_{l=1}^{kM}$,
where
$$
B_l=[i]_0\times\Bigl[\frac{j}{M}, \frac{j+1}{M}\Bigr]
$$
for $l=(i-1)M+j+1,1\leq i \leq k, 0\leq j\leq M-1$. Now we can define
a $kM\times kM$ transition matrix $A$ by
$$
A(s,t)=\left\{\begin{array}{ll}
1 \quad &\mbox{if}\;\tilde{\sigma}(\mbox{int}B_s)\supset \mbox{int}B_t\\
0\quad &\mbox{if}\;\;\tilde{\sigma}(\mbox{int}B_s)\cap
\mbox{int}B_t=\emptyset.
\end{array}\right.
$$
This will take the form
$$
A=\left(
         \begin{array}{lr}
              Q_1\\
              \;\vdots\\
              Q_k
          \end{array} \right),
$$
where $Q_s (1\leq s\leq k)$ is an $M\times kM$ matrix with the form
$$
Q_s=\left(
         \begin{array}{lr}
              P_s\cdots P_s\\
              \;\;\vdots \;\;\;\;\;\;\vdots\;\;\\
              P_s\cdots P_s
          \end{array} \right),
$$
in which $P_s$ is a $\prod\limits_{j\neq s}^{k}L_j\times M$ matrix
given by
$$
P_s=\left(
         \begin{array}{lr}
              1 \cdots 1\;\; 0 \cdots 0 \cdots 0\cdots 0\\
              0 \cdots 0\; \;1 \cdots 1 \cdots 0\cdots 0\\
              \;\;\cdots\; \;\;\;\cdots\;\;\;\;\cdots\;\;\;\;\cdots\\
              \underbrace{0 \cdots 0}_{\times L_s}\; \;\underbrace{0 \cdots 0}_{\times L_s} \cdots \underbrace{1\cdots 1}_{\times L_s}\\
          \end{array} \right).
$$
By calculating, we can get that $A$ is irreducible. Denote
$$
\Lambda=\{\{z_n\}_{n\in \mathbb{Z}_+}\in
\Sigma_{kM}\;:\;A(z_n,z_{n+1})=1\;\mbox{for}\;n\geq 0\}
$$
and let $\hat{\sigma}: \Lambda\rightarrow\Lambda$ denote the
associated subshift of finite type. We observe that the column sums
in $A$ are all equal to $\sum\limits_{i=1}^{k}L_i$. Therefore, by Perron-Frobenius Theorem and Theorem 7.13 of \cite{Walters}, we obtain that
\begin{equation}\label{equition}
h(\hat{\sigma})=\log\sum\limits_{i=1}^{k}L_i.
\end{equation}

Consider the map $\tilde{\pi}: \Lambda\rightarrow X_T$ defined by
$\tilde{\pi}\bigl(\{z_n\}\bigr)=\{x_n\}$ with
$x_0=\bigcap\limits_{m=0}^\infty I_m$ where
$$
I_m=\bigcap_{n=0}^{m-1}(T_{i_n}\circ\cdots\circ
T_{i_1})^{-1}\Bigl[\frac{l_n}{M}, \frac{l_n+1}{M}\Bigr]
$$
and $x_{n+1}=T_{i_n}(x_n)$, in which $T_{i_n}=T_t$ if the element in
the cover according to $z_n$ is $\Bigl([t]_0,
\Bigl[\displaystyle\frac{l_n}{M},
\displaystyle\frac{l_n+1}{M}\Bigr]\Bigr)$. Since $I_0\supset
I_1\supset\cdots\supset I_m\supset\cdots$ is a nested sequence of
closed sets we have by compactness that $\bigcap\limits_{m=0}^\infty
I_m\neq \emptyset$. Moreover, since each $L_i$ is greater than 1,
each $T_i$ is expanding and hence $\lim\limits_{m\longrightarrow
\infty}\mbox{diam }I_m=0$. In particular, this intersection consists
of a single point, say $x_0$. Therefore the map $\tilde{\pi}$ is
well defined. Similar to what we have done to $\pi$ in the proof of
Proposition \ref{prop3} we can show that $\tilde{\pi}$ is  a
semi-conjugacy, i.e.,
$\sigma_T\circ\tilde{\pi}=\tilde{\pi}\circ\hat{\sigma}$. Hence
\begin{equation}\label{inequality1}
h(\sigma_T)\leq h(\hat{\sigma}).
\end{equation}

It only remains to show that
$h(\sigma_T)\geq\log\sum\limits_{i=1}^{k}L_i$. Observe that although
$\tilde{\pi}$ is surjective, it can fail to be injective. We claim
that the set on which injectivity fails is ``small". Assume that
$\tilde{\pi}\bigl(\{z_n\}\bigr)=\tilde{\pi}\bigl(\{z_n'\}\bigr)=\{x_n\}$
but $\{z_n\}\neq\{z_n'\}$. In particular, assume that $z_i=z_i'$ for
$0\leq i\leq n-1$, but $z_n\neq z_n'$. This can only happen if
$x_n\in \Bigl\{\displaystyle\frac{i}{M} \;:\;0\leq i\leq M\Bigr\}$,
since $L_i, 1\leq i\leq k$ are pairwise different. In particular, we
see that
$$
\Omega=\bigl\{\{z_n\}\in\Lambda\; :\; \mbox{card}
(\tilde{\pi}^{-1}(\tilde{\pi}(\{z_n\})))\geq 2\bigr\}
$$
is a countable set.

Since $A$ is irreducible then $\hat{\sigma}: \Lambda\rightarrow\Lambda$ is a transitive
subshift of finite type. Hence from \cite{Parry}, there is a unique measure of maximal
entropy, i.e., $\nu$ is the unique $\hat{\sigma}$-invariant
probability measure with entropy
$h(\hat{\sigma})=\log\sum\limits_{i=1}^{k}L_i$. Moreover,
$\nu$ is the Markov measure, it is clear that $\nu(\Omega)=0$ and so
$\tilde{\pi}: \bigl(\Lambda, \nu\bigr)\rightarrow(X_T,
\tilde{\pi}^*\nu)$ is an isomorphism. By the variational principle
we see that
\begin{eqnarray}\label{inequality2}
h(\sigma_T)&=&\sup\bigl\{h_{\mu}(\sigma_T) : \mu\;\text{is
any}\; \sigma_T
\;\text{invariant probability measure}\bigr\}\notag \\
&\geq&h_{\tilde{\pi}^*\nu}(\sigma_T)=h(\hat{\sigma}).
\end{eqnarray}

Combining (\ref{equition}), (\ref{inequality1}) and (\ref{inequality2}) we
obtain the desired equality (\ref{circleentropy}).
\end{proof}

\section{$\mathbb{Z}_+^k$-actions on tori generated by expanding endomorphisms}

Let $A:\mathbb{R}^n\longrightarrow \mathbb{R}^n$ be a non-singular integer
matrix. There is a natural induced endomorphism of the $n$-dimensional torus
$\mathbb{T}^n=\mathbb{R}^n/\mathbb{Z}^n$, for simplicity, we also
denote it by $A$. It is well known that for any endomorphism $A$ on the torus $\mathbb{T}^n$,  we have
\begin{equation}\label{autonomous}
h(A)=\sum_{|\lambda^{(j)}|>1}\log |\lambda^{(j)}| ,
\end{equation}
where $\lambda^{(1)},\cdots,\lambda^{(n)}$ are the eigenvalues of
$A$, counted with their multiplicities.

In section 3, we use Geller and Pollicott's method to get a formula of Friedland's entropy for expanding ${\mathbb{Z}}_+^k$-actions on circles. However, it seems that it is not easy to use a similar strategy to deal with the high dimensional cases.  In this section, we use other entropy-like invariants, the so called
preimage entropies, to estimate the Friedland's entropy for the ${\mathbb{Z}}_+^k$-actions on tori generated by expanding endomorphisms. In the following, we first state some basic notions and facts for preimage entropies. For more information about them, please refer to \cite{Hurley}, \cite{Nitecki1} and \cite{Nitecki2}.

Let $f$ be a continuous surjective map on a compact metric space $(X,d_X)$. There are many types of preimage entropies and we only present two of them here. The first one is the
\emph{pointwise preimage entropy} $h_m(f)$ which is defined by
\begin{eqnarray*}
h_m(f)&=&\lim_{\varepsilon\rightarrow 0}\limsup
_{n\rightarrow \infty}\frac{1}
{n}\log\sup_{x\in X}s_{d_X}(f, n,\varepsilon,f^{-n}(x))\\&=&\lim_{\varepsilon\rightarrow 0}\limsup
_{n\rightarrow \infty}\frac{1}
{n}\log\sup_{x\in X}r_{d_X}(f, n,\varepsilon,f^{-n}(x)).
\end{eqnarray*}
The second one is the \emph{preimage branch entropy} which is
defined as follows. For any $x\in X$ and $n\in\mathbb{Z}_+$, the
$n$-th order preimage tree of $x$ under $f$ is defined by
$$
\mathcal{T}_{n}(x,f)=\{[z_n, z_{n-1}, \cdots, z_1, z_0=x]\;:\; f(z_j)=z_{j-1} \;\;\text{for all}\;\;1\leq j\leq n \}.
$$
Each ordered set $\xi=[z_n, z_{n-1}, \cdots, z_1, z_0=x]\in \mathcal{T}_{n}(x,f)$ is called a \emph{branch} of $\mathcal{T}_{n}(x,f)$. For any two branches
$$
\xi=[z_n, z_{n-1}, \cdots, z_1, z_0=x]\in \mathcal{T}_{n}(x,f)\;\;\text{and}\;\;\eta=[z_n', z_{n-1}', \cdots, z_1', z_0'=y]\in \mathcal{T}_{n}(y,f),
$$
the branch distance between them is defined as
$$
d_X^B(\xi, \eta)=\max_{0\leq j\leq n}d_x(z_j,z_j').
$$
Let $\mathcal{T}_{n}(X,f)=\bigcup\limits_{x\in X}\mathcal{T}_{n}(x,f)$. We can define a \emph{branch-Hausdorff metric} $d_X^b$ on $\mathcal{T}_{n}(X,f)$ by
\begin{eqnarray}\label{branchmetric}
 d_X^b(\mathcal{T}_{n}(x,f), \mathcal{T}_{n}(y,f))
=\max\{\max_{\xi\in \mathcal{T}_{n}(x,f)}\min_{\eta\in \mathcal{T}_{n}(y,f)}d_X^B(\xi, \eta),
\max_{\eta\in \mathcal{T}_{n}(y,f)}\min_{\xi\in \mathcal{T}_{n}(x,f)}d_X^B(\eta, \xi)\}
\end{eqnarray}
for $\mathcal{T}_{n}(x,f)$ and $\mathcal{T}_{n}(y,f)$ in
$\mathcal{T}_{n}(X,f)$. Intuitively, $d_X^b(\mathcal{T}_{n}(x,f),
\mathcal{T}_{n}(y,f))<\varepsilon$ if and only if each branch of
either tree is $d_X^B$ within $\varepsilon$ of at least one branch
of the other tree. Let
$s_{d_X^b}(f, n,\varepsilon,\mathcal{T}_{n}(X,f))$ denote the maximum
cardinality of any $d_X^b$-$\varepsilon$-separated collection of
trees in $\mathcal{T}_{n}(X,f)$, and
$r_{d_X^b}(f, n,\varepsilon,\mathcal{T}_{n}(X,f))$ denote the minimum
cardinality of any $d_X^b$-$\varepsilon$-spanning collection of
trees in $\mathcal{T}_{n}(X,f)$. Then the preimage branch entropy
$h_i(f)$ is defined by
\begin{eqnarray*}
h_i(f)&=&\lim_{\varepsilon\rightarrow 0}\limsup
_{n\rightarrow \infty}\frac{1}
{n}\log s_{d_X^b}(f, n,\varepsilon,\mathcal{T}_{n}(X,f))\\&=&\lim_{\varepsilon\rightarrow 0}\limsup
_{n\rightarrow \infty}\frac{1}
{n}\log r_{d_X^b}(f, n,\varepsilon,\mathcal{T}_{n}(X,f)).
\end{eqnarray*}
Similar to that for the entropy $h(f)$, it is easy to see that $h_m(f)$ and $h_i(f)$ are unchanged by taking an equivalent metric on $X$. From Theorem 3.1 in \cite{Hurley}, we have the following inequalities relating these entropies
\begin{equation}\label{inequality*}
h_m(f)\leq h(f)\leq h_m(f)+h_i(f).
\end{equation}
For some recent progress in the study of preimage entropies in different forms and in
different settings, we refer to \cite{Cheng}, \cite{Zhang},
\cite{Zeng}, \cite{Zhu2} and \cite{Zhu3}.

\begin{theorem}\label{torus}
Let $T$ be a ${\mathbb{Z}}_+^k$-action on the torus
$\mathbb{T}^n$ with the generators $\{T_i=A_i\}_{i=1}^k$ which are pairwise different.  If
$\{A_i\}_{i=1}^k$ are all non-singular and all eigenvalues of $A_i$
are of modulus greater than $1$, then the inequality (\ref{torusentropy}) holds, i.e.,
$$
h(T)\leq\log\Bigl(\sum_{i=1}^{k}\prod\limits_{j=1}^{n}|\lambda_i^{(j)}|\Bigr) ,
$$
where $\lambda_i^{(1)},\cdots,\lambda_i^{(n)}$ are the eigenvalues
of $A_i$, counted with their multiplicities.
\end{theorem}

\begin{proof}
Firstly, we show that
\begin{equation}\label{equality*}
h_i(\sigma_T)=0.
\end{equation}
 Since $\{A_i\}_{i=1}^k$ are all non-singular, for any $x\in \mathbb{T}^n$ and $1\leq i\leq k$,
 $$
\text{card}(T_i^{-1}(x))=|\det A_i|=\prod_{j=1}^{n}|\lambda_i^{(j)}|.
$$
For simplicity of notation, we denote $\prod\limits_{j=1}^{n}|\lambda_i^{(j)}|$ by $N_i$ and for $x\in \mathbb{T}^n$ denote
 $$
 T_i^{-1}(x)=\{x^{(1)},x^{(2)},\cdots,x^{(N_i)}\}.
 $$
 Since for each $1 \leq i\leq k$ all eigenvalues of $A_i$,
are of modulus greater than $1$, each $T_i$ is expanding. Hence, we can take $0<\varepsilon_0<\frac{1}{4}$ and $\lambda>1$, such that for any $0<\varepsilon\leq\varepsilon_0, x\in \mathbb{T}^n$ and $1\leq i\leq k$,
$$
T_i^{-1}(B_{d_X}(x,\varepsilon))=\bigcup_{j=1}^{N_i}U(x^{(j)}),
$$
where $U(x^{(j)})$ is an open neighborhood of $x^{(j)}$ for $1\leq j\leq N_i$, and $\{U(x^{(j)})\}_{j=1}^{N_i}$ are
pairwise disjoint, moreover, the restriction $T_i|_{U(x^{(j)})}: U(x^{(j)})\rightarrow B_{d_X}(x,\varepsilon)$ is a
homeomorphism and for any $z, z'\in U(x^{(j)})$,
$$
d_X(T_i(z), T_i(z'))\geq\lambda\cdot d_X(z,z').
$$
So for any $y\in B_{d_X}(x,\varepsilon)$, the corresponding $1$-th
order preimage tree $\mathcal{T}_1(y, T_i)$ lies in the
$d_X^b$-$\varepsilon$-neighborhood of the $1$-th order preimage tree
$\mathcal{T}_1(x, T_i)$ and actually
$$
d_X^b(\mathcal{T}_1(x, T_i), \mathcal{T}_1(y, T_i))=d_X(x, y)
$$
since $\lambda>1$. If for any finite sequence of endomorphisms $\bigl\{T_{n_i}\in
\{T_i, \cdots, T_k\}\bigr\}_{i=1}^l$ and any $x\in \mathbb{T}^n$,
let
$$
\mathcal{T}_l\bigl(x, \{T_{n_i}\}_{i=1}^l\bigr)=\bigl\{[z_l, z_{l-1}, \cdots, z_1, z_0=x]\;:\;T_{n_j}(z_j)=z_{j-1}\;\;\text{for all}\;\;1\leq j\leq l\bigr\}
$$
be the $l$-th order preimage tree of $x$ with respect to
$\{T_{n_i}\}_{i=1}^l$. Then for any $0<\varepsilon\leq\varepsilon_0$ and $y\in B_{d_X}(x,\varepsilon)$,
we can inductively conclude that the $l$-th order preimage tree
$\mathcal{T}_l\bigl(y, \{T_{n_i}\}_{i=1}^l\bigr)$ lies in the
$d_X^b$-$\varepsilon$-neighborhood $\mathcal{T}_l\bigl(x,
\{T_{n_i}\}_{i=1}^l\bigr)$, and actually
$$
d_X^b\bigl(\mathcal{T}_l(x, \{T_{n_i}\}_{i=1}^l),
\mathcal{T}_l(y, \{T_{n_i}\}_{i=1}^l)\bigr)=d_X(x,y),
$$
where $d_X^b$ is analogues to that in (\ref{branchmetric}).

From the above discussion and the definition of $d_{X_T}^b$ on the collection of $l$-th order preimage trees under $\sigma_T$, $\mathcal{T}_l(X_T, \sigma_T)$, we can see that
for any $\bar{x},\bar{y}\in X_T$ with $d(\bar{x},\bar{y})<\varepsilon\leq\varepsilon_0$, we have
$$
d_{X_T}^b\bigl(\mathcal{T}_l(\bar{x}, \sigma_T),
\mathcal{T}_l(\bar{y}, \sigma_T)\bigr)=d_{X_T}(\bar{x},\bar{y}).
$$
So, if a finite set
$\{\bar{x}^{(i)}\}$ is $d_{X_T}$-$\varepsilon_0$-dense in the compact space $X_T$, then for
any $l\in \mathbb{Z}_+, \bigl\{\mathcal{T}_l(\bar{x}^{(i)},
\sigma_T)\bigr\}$ is $d_{X_T}^b$-$\varepsilon_0$-dense in
$\mathcal{T}_l(X_T, \sigma_T)$. Therefore, $s_{d_{X_T}^b}(\sigma_T, l,
\varepsilon_0, \mathcal{T}_l(X_T, \sigma_T))$ is independent of $l$
and hence (\ref{equality*}) holds.

By (\ref{equality*}) and the inequalities in (\ref{inequality*}), we
have that
$$
h(T)=h(\sigma_T)=h_m(\sigma_T).
$$
Since for any $x\in \mathbb{T}^n$ and any $1\leq i\leq k$, card$(T_i^{-1}(x))=N_i$ and card$\bigl(\bigcup\limits_{i=1}^kT_i^{-1}(x)\bigr)\leq \sum\limits_{i=1}^kN_i$, and hence for any $l\in \mathbb{Z}_+$ and $\bar{x}=\{x_n\}_{n\in\mathbb{Z}^+}\in X_T$, card$(\sigma_T^{-l}(\bar{x}))\leq\Bigl(\sum\limits_{i=1}^{k}N_i\Bigr)^l$. Therefore,
\begin{equation}\label{h_m}
h_m(\sigma_T)\leq\log\sum\limits_{i=1}^kN_i,
\end{equation}
and then the desired inequality (\ref{torusentropy}) holds.
\end{proof}

By the way, we would like to say that our original intention is to show that the equality in  (\ref{torusentropy}) holds for expanding ${\mathbb{Z}}_+^k$-actions on tori. So far we can conclude as follows that for almost every $\bar{x}=\{x_n\}_{n\in\mathbb{Z}^+}\in X_T$, card$(\sigma_T^{-l}(\bar{x}))=\Bigl(\sum\limits_{i=1}^{k}N_i\Bigr)^l$ for any $l\in \mathbb{Z}_+$.
Let
$$
E=\{x\in \mathbb{T}^n:T_ix=T_jx\;\mbox{for some}\;1\leq i,j\leq k, i\neq j\}
$$
and
$$
F=\bigcup_{n=1}^{\infty}\bigcup_{1\leq i_1,\cdots,i_n\leq k}T_{i_1}\circ\cdots\circ T_{i_n}(E).
$$
 Since $\{A_i\}_{i=1}^k$ are all non-singular and pairwise different, then the Haar measure of $E$ is zero, and hence the Haar measure of $F$ is also zero.
So for $x\in \mathbb{T}^n\setminus F$, the cardinality of the set $\bigcup\limits_{i=1}^kT_i^{-1}(x)$ is exactly $\sum\limits_{i=1}^kN_i$. Therefore, for any $l\in \mathbb{Z}_+$ and $\bar{x}=\{x_n\}_{n\in\mathbb{Z}^+}\in X_T$ with $x_0\in \mathbb{T}^n\setminus F$, the cardinality of the $l$-th preimage set $\sigma_T^{-l}(\bar{x})$ is exactly $\Bigl(\sum\limits_{i=1}^{k}N_i\Bigr)^l$. We believe that for any $\bar{x}=\{x_n\}_{n\in\mathbb{Z}^+}\in X_T$ with $x_0\in \mathbb{T}^n\setminus F$ (even for any $\bar{x}=\{x_n\}_{n\in\mathbb{Z}^+}\in X_T$),
$$
\lim_{\varepsilon\longrightarrow 0}\limsup_{l\rightarrow \infty}\frac{1}
{l}\log s_{d_{X_T}}(\sigma_T,l,\varepsilon,\sigma_T^{-l}(\bar{x}))=\log\Bigl(\sum_{i=1}^{k}N_i\Bigr),
$$
and hence the equality in  (\ref{torusentropy}) holds.

 \begin{remark}
In \cite{Zhu},  the upper and lower bounds of the entropy of the nonautonomous dynamical systems on tori which are generated by expanding endomorphisms (Theorem 2.8 of \cite{Zhu}) were given. More precisely, let $A_{1,\infty}=\{A_i\}_{i=1}^{\infty}$ be a sequence of
equi-continuous surjective endomorphisms of ${\mathbb{T}}^n$. If for each
$i\in \mathbb{Z}_+$, all eigenvalues of $A_i$ are of modulus
greater than $1$, then
\begin{equation}\label{inequa6}
\limsup_{l\rightarrow\infty}
\frac{1}{l}\sum\limits_{i=1}^{l-1}\sum\limits_{j=1}^{n}\log|\lambda_i^{(j)}|
\leq h(A_{1,\infty})\leq\limsup_{n\rightarrow\infty}
\frac{n}{l}\sum\limits_{i=1}^{l-1}\log\Lambda_i^{(1)},
\end{equation}
where $\lambda_i^{(1)},\cdots,\lambda_i^{(n)}$ are the eigenvalues
of $A_i, i\in\mathbb{Z}_+$, counted with their multiplicities,
and $\Lambda_i^{(1)}$ is the biggest eigenvalue of
$\sqrt{A_iA_i^T}, i\in\mathbb{Z}_+$. From Theorem 3.1 and Theorem 5.1 in \cite{Zhang}, we have that
$h(A_{1,\infty})=h_m(A_{1,\infty})$.  Using the similar method in the proof of the above Theorem \ref{torus}, we can improve (\ref{inequa6}) to the following equality
$$
h(A_{1,\infty})=h_m(A_{1,\infty})=\limsup_{l\rightarrow\infty}
\frac{1}{l}\sum\limits_{i=1}^{l-1}\sum\limits_{j=1}^{n}\log|\lambda_i^{(j)}|.
$$
 \end{remark}

\section*{Acknowledgements} The authors would like to
thank the referee for the detailed review and very good suggestions,
which led to improvements of the paper.


\begin{thebibliography}{99}

\bibitem {Bowen}   R.Bowen, \emph{Entropy for group endomorphisms and homogenuous spaces},
                Trans. Amer. Math. Soc., 153(1971), 401-414.

\bibitem {Cheng} W.-C.Cheng and S.Newhouse, \emph{Pre-image entropy}, Ergod. Th. and Dynam. Sys., 25 (2005), 1091-1113.

\bibitem {Einsiedler} M.Einsiedler and D.Lind,  \emph{Algebraic $\mathbb{Z}^d$-actions on entropy rank one}, Trans. Amer. Math. Soc., 356(5)(2004), 1799-1831.

\bibitem{Friedland} S.Friedland, \emph{Entropy of graphs, semi-groups and groups}, in: Ergodic Theory of
$\mathbb{Z}^d$-actions, M.Pollicott and K.Schmidt (eds.), London
Math. Soc. Lecture Note Ser. 228, Cambridge Univ. Press, Cambridge,
(1996), 319-343.

\bibitem{Geller} W.Geller and M.Pollicott, \emph{An entropy for $\mathbb{Z} ^2$ -actions with finite entropy generators}, Fund. Math., 157(1998), 209-220.

\bibitem{Huang} W.Huang, X.Ye and G.Zhang, Local entropy theory for a countable
discrete amenable group action, J. Funct. Anal., 261(2011), 1028-1082.

\bibitem {Hurley} M.Hurley, \emph{On topological entropy of maps}, Ergod. Th. and Dynam. Sys., 15(1995), 557-568.

\bibitem{Katok} A.Katok and B.Hasselblatt, Introduction to the
modern theory of dynamical systems.
Cambridge University Press, Cambridge, 1995.

\bibitem {Nitecki1} Z.Nitecki and F.Przytycki, \emph{Preimage entropy for mappings}, Int. J. Bifur. and Chaos, 9 (1999), 1815-1843.

\bibitem {Nitecki2}  Z.Nitecki, \emph{Topological entropy and the preimage structure of maps},
                Real Analysis Exchange, 29(2003/2004), 7-39.

\bibitem{Ornstein} D.Ornstein and B.Weiss, Entropy and isomorphism theorems for actions of amenable
groups, J. Anal. Math., 48(1987), 1-141.

\bibitem{Parry} W.Parry, Intrinsic Markov Chains. \emph{Trans. Amer. Math. Soc.,} 112(1964), 55-66.

\bibitem{Ruelle} D.Ruelle, \emph{Statistical mechanics
on a compact set with $\mathbb{Z}^{\nu}$-action satisfying expansiveness and
specification}, Trans. Amer. Math. Soc., 185(1973), 237-251.

\bibitem{Schmidt} K.Schmidt, Dynamical systems of algebraic origin, New York, Berlin:
Birkhauser-Verlag, 1995.

\bibitem{Walters} P.Walters, An introduction to ergodic theory, Springer, New York, 1982.

\bibitem {Zeng} F.Zeng, K.Yan and G.Zhang \emph{Pre-image pressure and invariant measures}, Ergod. Th. and Dynam. Sys., 27(2007), 1037-1052.

\bibitem {Zhang} J.Zhang, Y.Zhu and L.He, \emph{Preimage entropy for nonautonomous dynamical
                systems}, Acta Math. Sinica, Chinese Series, 48(2005), 693-702.


\bibitem {Zhu2} Y.Zhu, \emph{Preimage entropy for random dynamical
systems}, Discrete and Contin. Dynam. Sys., 18(2007), 829-851.

\bibitem {Zhu3} Y.Zhu, Z.Li  and X.Li, \emph{Preimage pressure for random transformations}, Ergod. Th. and Dynam. Sys., 29(2009), 1669-1687.

\bibitem{Zhu} Y.Zhu, Z.Liu, X.Xu and W.Zhang, \emph{Entropy of Nonautonomous Dynamical
Systems}, J. Korean Math. Soc.,  49(2012), 165-185.







\end{thebibliography}
\end{document}